\documentclass[12pt,a4paper,reqno]{amsart}
\usepackage[margin=2cm]{geometry}

\usepackage{graphics, amsmath,amssymb, enumitem, comment, url}
\usepackage{graphicx}
\usepackage{epsfig}

\newif\ifcolorcomments
\newcommand{\allowcomments}[4]{
\newcommand{#1}[1]{\ifdraft{\ifcolorcomments{\textcolor{#4}{##1 --#3}}\else{\textsl{ ##1 \ --#3}}\fi}\else{}\fi}
}

\allowcomments{\commumtaz}{MH}{Mumtaz}{green}
\allowcomments{\comjohannes}{JS}{Johannes}{blue}
\allowcomments{\comdavid}{DS}{DS}{magenta}

\colorcommentstrue
\usepackage{amssymb}
\usepackage{amsmath,amsthm}

\usepackage{colortbl}

\usepackage{amssymb,amsmath}
\usepackage{geometry}    
\usepackage{mathrsfs}

\newtheorem{theorem}{Theorem}[section]
\newtheorem{lemma}[theorem]{Lemma}
\newtheorem{proposition}[theorem]{Proposition}
\newtheorem{corollary}[theorem]{Corollary}
\theoremstyle{definition}
\newtheorem{definition}[theorem]{Definition}
\newtheorem{remark}[theorem]{Remark}

\newcommand{\DD}{\mathcal D}

\newcommand{\HH}{\mathcal H}

\newcommand{\MM}{\mathcal M}

\newcommand{\N}{\mathbb N}

\newcommand{\Q}{\mathbb Q}

\newcommand{\R}{\mathbb R}

\newcommand{\UU}{\mathcal U}

\newcommand{\VV}{\mathcal V}

\newcommand{\Z}{\mathbb Z}

\newcommand{\mbf}{\mathbf}
\newcommand{\0}{\mbf 0}
\newcommand{\ba}{\mbf a}

\newcommand{\ee}{\mbf e}

\newcommand{\hh}{\mbf h}

\newcommand{\qq}{\mbf q}
\newcommand{\rr}{\mbf r}

\newcommand{\vv}{\mbf v}
\newcommand{\ww}{\mbf w}
\newcommand{\xx}{\mbf x}
\newcommand{\yy}{\mbf y}
\newcommand{\zz}{\mbf z}

\renewcommand{\text}{\textup}

\newcommand{\NPC}[1]{\ignorespaces}

\newif\ifdraft\drafttrue
\IfFileExists{marvosym.sty}{
\RequirePackage{marvosym} 
}{

}

\IfFileExists{wasysym.sty}{
\RequirePackage{wasysym}\renewcommand{\emptyset}{{\diameter}}
}{

}

\newcommand{\DDD}{\DD_n^{\theta}(\Psi)}

\newtheorem*{GBSP1}{Generalised Baker-Schmidt Problem for Hausdorff Measure: dual setting}

\newcommand{\Qp}{\mathbb{Q}_{p}}
\newcommand{\Zp}{\mathbb{Z}_{p}}

\begin{document}
\title{Dual $p$-adic Diophantine approximation on manifolds}

\author[Mumtaz Hussain]{Mumtaz Hussain}
\address{Mumtaz Hussain,  Department of Mathematical and Physical Sciences,  La Trobe University, Bendigo 3552, Australia. }
\email{m.hussain@latrobe.edu.au}

\author[Johannes Schleischitz]{Johannes Schleischitz} 
\address{Johannes Schleischitz, Middle East Technical University, Northern Cyprus Campus, Kalkanli, G\"uzelyurt}
\email{johannes@metu.edu.tr ; jschleischitz@outlook.com}

\author{ Benjamin Ward}
\address{Ben Ward,  Department of Mathematical and Physical Sciences,  La Trobe University, Bendigo 3552, Australia. }
\email{Ben.Ward@latrobe.edu.au}

\begin{abstract} 
The Generalised Baker-Schmidt Problem (1970)  concerns the Hausdorff measure of the set of $\psi$-approximable points on a nondegenerate manifold. Beresnevich-Dickinson-Velani (in 2006, for the homogeneous setting) and Badziahin-Beresnevich-Velani (in 2013,  for the inhomogeneous setting) proved the divergence part of this problem for dual approximation on arbitrary nondegenerate manifolds. The divergence part has also been resolved for the $p$-adic setting by Datta-Ghosh in 2022 for the inhomogeneous setting. The corresponding convergence counterpart represents a challenging open problem. In this paper, we prove the homogeneous $p$-adic convergence result for hypersurfaces of dimension at least three with some mild regularity condition, as well as for some other classes of manifolds satisfying certain conditions. We provide similar, slightly weaker results for the inhomogeneous setting. We do not restrict to monotonic approximation functions. 
\end{abstract}
\maketitle

\maketitle

\section{Dual Diophantine approximation on manifolds}

Throughout, let $n\geq 1$ be a fixed integer and $\qq:=(q_1,\ldots, q_n)\in \Z^n$. Let $\Psi:\Z^n\to[0, \infty)$ be a \emph{multivariable approximating function}, that is,  $\Psi$ has the property that
$$\Psi(\qq) \rightarrow~0 \text{ as } \|\qq\|:=\max(|q_1|, \ldots, |q_n|) \rightarrow~\infty.$$ For $\theta\in \R$, consider the set
\begin{equation*}
  \DDD:=\left\{\xx\in\R^n:\begin{array}{l}
  |\qq\cdot\xx+p+\theta|<\Psi(\qq)  \; \text{for infinitely many} \ (p, \qq)\in\Z\times \Z^{n}
                           \end{array}
\right\},
\end{equation*}
where $\xx=(x_1,\dots,x_n)\in\R^n$ so that
\begin{equation*}
\qq\cdot\xx=q_{1}x_{1}+\dots +q_{n}x_{n}.
\end{equation*}
 A vector $\xx \in \R^n$ will be called dually \emph{$(\Psi, \theta)$-approximable} if it lies in the set $\DDD$. The set $\DDD$ corresponds to \emph{inhomogeneous} Diophantine approximation, and when $\theta= 0$ the problem reduces to \emph{homogeneous} approximation. In this case we write $\DD_{n}^{0}(\Psi):=\DD_{n}(\Psi)$. When $\Psi$ is of the form $\Psi(\qq)=\psi(\|\qq\|)$ for some norm $\|\cdot\|$ and $\psi:[0,\infty) \to[0,\infty)$, we say $\psi$ is a univariable approximation function. \par 
We are interested in the `size' of the set $\DDD$ with respect to $f$-dimensional Hausdorff measure $\HH^f$ for some dimension function $f$. By a dimension function $f$ we mean an increasing continuous function $f:\mathbb{R}_{+}\to \mathbb{R}_{+}$ with $f(0)=0$. When $f(r)=r^{n}$ then $\HH^f$-measure is comparable to $n$-dimensional Lebesgue measure.

A classical result due to Dirichlet tells us that $\DD_{n}(\psi_{n})=\R^{n}$ for $\psi_{n}(r)=r^{-n}$, and an application of the Borel Cantelli lemma informs us that the set of \emph{very well approximable} points (there exists $\epsilon>0$ such that $\xx \in \DD_{n}(\psi_{n+\epsilon})$)  is a Lebesgue nullset. Furthermore, the generalised set of inhomogeneous very well approximable points is also a nullset. \par
 The following theorem due to Schmidt summarises the Lebesgue measure theory.

\begin{theorem}[Schmidt \cite{Sch64}] \label{Schmidt}
Fix $\theta \in \R$ and suppose $n\geq 2$. Let $\Psi:\Z^{n} \to [0,\infty)$ be a multivariable approximation function. Then
\begin{equation*}
\lambda_{n}\left( \DD_{n}^{\theta}(\Psi) \right) = \begin{cases}
0 \quad \quad \text{ if } \, \sum\limits_{\qq \in \Z^{n}} \Psi(\qq) < \infty, \\[2ex]
{\rm full} \quad \text{ if } \, \sum\limits_{\qq \in \Z^{n}} \Psi(\qq) = \infty.
\end{cases}
\end{equation*}
\end{theorem}

Here by ``{\rm full}" we mean the complement is a nullset. Note that for $n=1$, the monotonicity of $\Psi$ is required in the above statement, since even in the homogeneous case the result is known to be false due to the work of Duffin and Schaeffer, see for example \cite[Theorem 2.8]{Har98}. It should be remarked this theorem contains the notable Khintchine-Groshev Theorem ($\theta=0$ and $\Psi(\qq)=\psi(\|\qq\|)$ monotonic). \par 
Using a `slicing' technique and the mass transference principle, Beresnevich and Velani extended Schmidt's theorem to the Hausdorff measure statement which reads as follows.

\begin{theorem}[Beresnevich and Velani \cite{BV06b}] \label{SBV}
Fix $\theta \in \R$ and suppose $n\geq2$. Let $\Psi:\Z^{n}\to[0,\infty)$ be a multivariable approximation function and let $f$ be a dimension function such that $r^{-n}f(r)\to \infty$ as $r \to 0$. Assume $r^{-n}f(r)$ is decreasing and $r^{-n+1}f(r)$ is increasing. Then
\begin{equation*}
\HH^{f}\left(\DD_{n}^{\theta}(\Psi)\right)= \begin{cases}
0 & \text{ if } \, \sum\limits_{\qq \in \Z^{n}} \|\qq\|^{n}\Psi(\qq)^{1-n}f\left(\frac{\Psi(\qq)}{\|\qq\|}\right) < \infty, \\[2ex]
\infty &\text{ if } \, \sum\limits_{\qq \in \Z^{n}} \|\qq\|^{n}\Psi(\qq)^{1-n}f\left(\frac{\Psi(\qq)}{\|\qq\|}\right)  = \infty.
\end{cases}
\end{equation*}
\end{theorem}
Again, note that this result contains Jarnik's Theorem (in the dual setting), and one can deduce the dual version of the inhomogeneous Jarnik-Besicovich Theorem due to Levesley \cite{Lev98}.

\subsection{Diophantine approximation on manifolds}
In 1932, Mahler initiated the study of Diophantine approximation on dependent quantities by conjecturing that the set of very well approximable points on the Veronese curve $\VV_{n}=\{(x,x^{2},\dots ,x^{n}):x\in \R\}$ is a nullset with respect to the induced Lebesgue measure on $\VV_{n}$. That is, Mahler conjectured that
\begin{equation*}
\left\{ x \in \R :  \exists \, \varepsilon>0 \text{ such that } \begin{array}{c} |q_{1}x +q_{2}x^{2} + \dots + q_{n}x^{n} -p|<\left(\max_{1\leq i \leq n}|q_{i}|\right)^{-n-\varepsilon} \\
\text{for infinitely many}  (p,q_{1},\dots , q_{n})\in \Z^{n+1}
\end{array} \right\}
\end{equation*}
is a Lebesgue nullset. Sprindzhuk proved Mahler's conjecture to be true and conjectured that the result remained true for any analytic nondegenerate manifold \cite{Spr69}. It was not until the fundamental work of Kleinbock and Margulis \cite{KM98} that Sprindzhuk’s conjecture was proven true\footnote{ The main results of this paper are exclusively in the dual setting, but it should be noted that the result of Kleinbock and Margulis was proved in the generalised multiplicative setting.}. Since this breakthrough paper, a range of results in quick succession have been proven in this area. We provide a brief survey of key results, which can be split into three parts: \par
\textit{Extremality} refers to results associated with Mahler's (and subsequently Sprindzhuk's) conjecture. Specifically the size, in terms of the induced Lebesgue measure, of the set of very well approximable points contained in some manifold is a nullset. A manifold is said to be extremal if it satisfies Sprindzhuk's conjecture.\par 
\textit{Ambient measure} results refer to analogues of Theorem~\ref{Schmidt} in the setting of Diophantine approximation on dependent quantities, that is the induced Lebesgue measure of $\DD_{n}^{\theta}(\Psi)\cap \MM$. \par 
\textit{Hausdorff theory} results refer to the Hausdorff measure and dimension of  $\DD_{n}^{\theta}(\Psi)\cap \MM$. In full generality, a complete Hausdorff measure treatment akin to Theorem~\ref{SBV} for manifolds $\MM$  represents a deep open problem referred to as the Generalised Baker-Schmidt Problem (GBSP) inspired by the pioneering work of Baker and Schmidt \cite{BakerSchmidt}. 

\begin{GBSP1}
Let $\MM$ be a nondegenerate submanifold of $\R^n$ with $\dim\MM=d$ and $n \geq 2$. Let $\Psi$ be a multivariable approximating function.  Let $f$ be a dimension function such that $r^{-d}f(r)\to\infty$ as $r\to 0.$ Assume that $r\mapsto r^{-d}f(r)$ is decreasing and $r\mapsto r^{1-d}f(r)$ is increasing. Prove that 

\begin{equation*}
  \HH^f( \DDD\cap\MM)=\left\{\begin{array}{cl}
 0 &  {\rm if } \quad\sum\limits_{\qq\in\Z^n\setminus \{\0\}}\|\qq\|^d\Psi(\qq)^{1-d }f\left(\frac{\Psi(\qq)}{\|\qq\|}\right)< \infty,\\[3ex]
 \infty &  {\rm if } \quad \sum\limits_{\qq\in\Z^n\setminus \{\0\}}\|\qq\|^d\Psi(\qq)^{1-d}f\left(\frac{\Psi(\qq)}{\|\qq\|}\right)=\infty.
                                     \end{array}\right.
\end{equation*}
\end{GBSP1}

Note the results of extremality and ambient measure would follow from a result of the above form. For brevity, the following results are stated with few details. In some cases, further technical details, especially properties on the approximation function $\Psi$ and dimension function $f$, are required for the statement given to be true.

\begin{itemize}
\item \textit{Extremality}: Sprindzhuk proved this for the Veronese curve \cite{Spr69}, Kleinbock and Margulis proved this for nondegenerate manifolds \cite{KM98}. In the inhomogeneous setting Badziahin proved the result for nondegenerate planar curves \cite{Bad10}, and Beresnevich and Velani extended the work of Kleinbock and Margulis to the inhomogeneous setting for nondegenerate manifolds \cite{BV10}.
\item \textit{Ambient measure}: The convergence case was proven for nondegenerate manifolds independently by Bernik, Kleinbock, and Margulis \cite{BKM01} and Beresnevich \cite{Ber02}. The complimentary divergence case was proven by Beresnevich, Bernik, Kleinbock and Margulis \cite{BBKM02}. The complete inhomogeneous version was proven by Beresnevich, Badziahin and Velani \cite{BBV13}.
\item \textit{Hausdorff theory}: In terms of the Hausdorff dimension the upper and lower bounds for the Veronese curve were proven by Bernik \cite{Ber83} and Baker $\&$ Schmidt \cite{BakerSchmidt} respectively. A lower bound for extremal manifolds was proven by Dickinson and Dodson \cite{DD00} and the complimentary upper bound was proven by Beresnevich, Bernik, and Dodson \cite{BBD02}. The Hausdorff measure convergence case was proven by the first name author for $\VV_{2}$ \cite{Hus15} and was later generalised by Huang \cite{Hua18} to nondegenerate planar curves. Partial results for the convergence case have been proven in higher dimensions. For various classes of curves see \cite{HSS2}, for hypersurfaces with non-zero Hessian and dimension $d\geq 3$ see \cite{HSS1}, and for more general manifolds with further restrictions on the curvature, see \cite{HS3}. Note the latter two papers also provide results in the inhomogeneous setting. The homogeneous divergence case for nondegenerate manifolds was proven completely by Beresnevich, Dickinson, and Velani \cite{BDV06}. The inhomogeneous divergence case was proven in \cite{BBV13}.
\end{itemize}

\subsection{$p$-adic Diophantine approximation}

Much of the setup presented in the real setting can be transferred to $p$-adic space. Fix a prime $p$ and let $|\cdot|_{p}$ denote the $p$-adic norm, $\Qp$ the $p$-adic numbers and $\Zp$ the ring of $p$-adic integers, that is, $\Zp:=\{x\in \Qp : |x|_{p}=1\}$. 

For a multivariable approximation function $\Psi:\Z^{n+1} \to \R_{+}$ and fixed $\theta \in \Qp$ define the set of $p$-adic dually $(\Psi,\theta)$-approximable points as
\begin{equation*}
D_{n, p}^{\theta}(\Psi):=\left\{\xx=(x_1,\dots,x_n)\in\Q_{p}^n:\begin{array}{l}
|a_1x_1+\cdots+a_nx_n+a_0+\theta|_{p}<\Psi(\ba) \\[1ex]
\text{for infinitely many} \  \ba=(a_0, \ldots, a_n)\in\Z^{n+1}
\end{array}
\right\}.
\end{equation*}
Similar to the real case, for the homogeneous setting write $D_{n, p}^{0}(\Psi):=D_{n, p}(\Psi)$, and for $\Psi$ a univariable approximation function of the form $\Psi(\rr)=\psi(\|\rr\|)$ write $D_{n, p}^{\theta}(\Psi):=D_{n, p}^{\theta}(\psi)$. \par 
Note that, unlike the real setting, $\Psi$ depends on the $n+1$ integers $(a_{0},\dots, a_{n})$, including $a_0$. This is because $\Z$ is dense in $\Zp$ and so, if $a_{0}$ is unbounded from above, one could obtain increasingly precise approximations of $\xx$ for $(a_{1},\dots, a_{n})$ fixed, at least if $\xx\in\Z_p^n$ and $ \theta\in \Z_p$.  \par
The concept of Hausdorff dimension and more generally Hausdorff $f$-measure for any dimension function $f:\R_{+}\to \R_{+}$ as in the real case can be defined similarly over $\mathbb{Q}_{p}^{n}$ by coverings of balls with respect to the $p$-adic metric derived from $\| \cdot\|_{p}=\max|\cdot|_{p}$. We refer to \cite{Rogers} for properties of the Hausdorff measure and dimension in general metric space. Thus it makes sense to ask questions about the `size' of $D_{n}^{\theta}(\Psi)$ with respect to the $f$-dimensional Hausdorff measure $\HH^{f}$ for some dimension function $f$. When $f(r)=r^{n}$ this reduces to the size of $D_{n}^{\theta}(\Psi)$ in terms of the $n$-dimensional $p$-adic Haar measure $\mu_{p,n}$ up to some fixed constant.

In the homogeneous setting Mahler,  see for example \cite{Mah73}, proved for $\Psi(\ba)=\psi(\|\ba\|)=\|\ba\|^{-(n+1)}$ that $D_{n,p}(\Psi)=\Qp^{n}$, providing a $p$-adic equivalent of Dirichlet's Theorem. Again, by an application of the Borel Cantelli lemma we can deduce that the set of $p$-adic very well approximable points is a nullset, where the measure here is the $n$-dimensional $p$-adic Haar measure $\mu_{p,n}$. Lutz proved the $p$-adic analogue of Khintchine's Theorem \cite{Lut55}, and the Hausdorff measure analogue can be found in \cite[Theorem 16]{BDV06}.

\subsection{$p$-adic Diophantine approximation on manifolds and our main result}   \label{padon}

For $p$-adic approximation on dependent quantities, the following results are known. Again, we keep the statements of the known results brief and refer the reader to the relevant paper for more details. A key additional condition often required in the $p$-adic setting is the analyticity of the curve or manifold, see Section~\ref{prelim} for reasoning as to why. Here the extremality and ambient measure statements are with respect to $\mu_{p,n}$ and the induced $p$-adic Haar measure on a manifold $\MM$.
\begin{itemize}
\item \textit{Extremality}: Alongside the statement in the real setting, Sprindzhuk proved the $p$-adic equivalent of Mahler’s conjecture in the $p$-adic setting \cite{Spr69}. Kleinbock and Tomanov \cite{KT07}, used similar ideas to those in \cite{KM98} to prove extremality for all $C^{2}$ nondegenerate manifolds (see \cite{KT07} for the precise definition of a $C^{2}$ function in the $p$-adic setting). The inhomogeneous theory for the Veronese curve preceded that of Kleinbock and Tomanov and was proven by Bernik, Dickinson, and Yuan \cite{BDY99}.
\item \textit{Ambient measure}: The complete theory for the Veronese curve was proven by Beresnevich, Bernik and Kovalevskaya \cite{BBK05}. Preceding this, in \cite{BK03}, Beresnevich and Kovalevskaya proved the complete result for $p$-adic normal\footnote{see \cite{Mah73} for Mahler's definition of a $p$-adic normal function.} planar curves. In the inhomogeneous setting the convergence case was proven for the Veronese curve by Ustinov \cite{Ust05}. The convergence case for analytic nondegenerate manifolds was proven by Mohammadi and Salehi-Golsefidy \cite{MSg11} with the complementary divergence statement appearing soon after \cite{MSg12}. The inhomogeneous convergence statement was proven in \cite{DattaGhosh}.
\item \textit{Hausdorff theory}: 
In the special case of Veronese curves, the metric theory is rather
complete by results of Bernik and Morotskaya \cite{bemo, moro}.
For a general class of manifolds, 
the divergence statement in both the homogeneous and inhomogeneous setting has recently been proven by Datta and Ghosh \cite{DattaGhosh}, see Section~\ref{DatGho} for their result. In this paper, we contribute to the inhomogeneous convergence case. 
\end{itemize}

For a subset $Z\subseteq \Z^{n+1}$ define the set
\begin{equation*}
D_{n, p}^{\theta}(\Psi,Z):=\left\{\xx=(x_1,\dots,x_n)\in\Q_{p}^n:\begin{array}{l}
|a_1x_1+\cdots+a_nx_n+a_{0}+\theta|_{p}<\Psi(\ba) \\[1ex]
\text{for infinitely many}  \ (a_{0}, a_1, \ldots, a_n)\in Z
\end{array}
\right\}.
\end{equation*}

For $Z=\Z^{n+1}\backslash\{\0\}$ we write $D^{\theta}_{n,p}(\Psi,Z)=D^{\theta}_{n,p}(\Psi)$. Let $\MM\subset \Qp^{n}$ be a $p$-adic manifold with dimension $d$ defined by analytic map $\textbf{g}:\UU\subset \Qp^{d} \to \Qp^{n-d}$ via parametrization
$(\xx,\textbf{g}(\xx))$. Hereby analytic is defined
as follows.

\begin{definition}  \label{Defana}
	A function $\textbf{h}: U\subseteq \mathbb{Q}_{p}^{m}\to \mathbb{Q}_{p}^{n}$ 
	for $U$ open is analytic
	if every coordinate function can be written as a power series
	\[
	h_{j}(x_{1},\ldots,x_{m})= \sum_{\boldsymbol{t}} 
	a_{j,\boldsymbol{t}} x_{1}^{t_{1}} \cdots x_{m}^{t_{m}}, \qquad 1\leq j\leq n,
	\]
	converging in some $p$-adic ball $B(\yy,r), \  r=r(\yy)>0$ 
	around every $\yy$ contained in $U$.
\end{definition}

Note that any analytic function $\textbf{h} \in C^{\infty}(U)$. Compared to the real setting where $\textbf{g}\in C^2$ was sufficient, for technical
reasons we require the stronger condition of analytic manifolds in the $p$-adic setting.

By the Implicit Function Theorem (the $p$-adic version of this following from $\textbf{g}$ being analytic and the $p$-adic inverse function theorem, see Theorem~\ref{ift}) we may write
\begin{equation*}
\MM=\left\{ (x_{1}, \dots , x_{d}, g_{1}(\xx), \dots , g_{n-d}(\xx)) : \xx=(x_{1},\dots,x_{d}) \in \UU \subset \Qp^{d} \right\}
\end{equation*}
 for analytic functions $g_{i}:\UU\subset \Qp^{d} \to \Qp$, $1\leq i\leq n-d$. For ease of notation write $\textbf{g}=(g_{1},\dots, g_{n-d})$.\par

Analogously to~\cite{HSS1, HS3}, minding the slightly modified notation, we impose the following conditions :
\begin{itemize}
\item[(Ip)] Let $f$ be a dimension function satisfying 
\begin{equation}\label{dim_func_cond}
f(xy) \ll x^{s}f(y) \text{ for all } y<1<x
\end{equation}
for some $s<2(d-1)$. \\
\item[(IIp)] For each $1\leq i \leq n-d$, let $g_{i}: \UU \to \Qp$ be analytic on some open set $\UU \subset \Qp^{d}$ and suppose that for any rational integer vector $\zz=(z_1,\dots,z_{n-d}) \in \Z^{n-d}$ with $\Vert \zz\Vert_p=1$, the $d\times d$ matrix with $p$-adic entries
\begin{equation*}
M_{\zz}(\xx)=\left( \sum_{k=1}^{n-d}z_{k}\frac{\partial^{2} g_{k}(\xx)}{\partial x_{i} \partial x_{j}} \right)_{1\leq i, j \leq d}
\end{equation*}
has non-zero determinant for all $\xx \in \UU\setminus S_{\MM}(\zz)$, except possibly on a set $$S_{\MM}(\zz):=\{\xx\in\UU: M_{\zz}(\xx) \ \text{is singular}\}$$ with $\HH^{f}(S_{\MM}(\zz))=0$.
\end{itemize}

For an example of a dimension function $f$ satisfying (Ip) one can consider the function $f(r)=r^{s}$ for some $0<s<2(d-1)$. Note that
$s>d$ is not of interest as the entire manifold
$\MM$ has $\HH^{s}(\MM)=0$. 
Condition (Ip) is always true when $d\ge 3$
and $r^{-d}f(r)$ is decreasing (a general assumption in the GBSP statement), see~\cite[Section~1.2]{HSS1} for hypersurfaces,
which can readily be generalised as remarked in \cite{HS3}. \par 

Condition (IIp) has the most natural interpretation
when $\MM$ is a hypersurface. Then (IIp) is equivalent to asking for the Hessian of $\textbf{g}=g:\UU\subset \Qp^{n-1} \to \Qp$, denoted by $\nabla^{2}g$, to be singular only 
on a set of $\HH^{f}$ measure zero. That is,
\begin{equation} \label{hypersurface statement}
    \HH^{f}\left(\left\{ \xx \in \UU : \nabla^{2} g(\xx) \text{ is singular} \right\} \right)=0\, .
\end{equation}
 Combining with what was noticed above for the (Ip) statement, we have that (Ip) and (IIp) are satisfied for any hypersurface of dimension at least three (thus ambient
space has a dimension at least four) satisfying \eqref{hypersurface statement}. For $\MM$ not
a hypersurface, a detailed discussion on condition (IIp), in the real case, is provided 
in \cite{HS3}. The $p$-adic case is similar. In short, in \cite{HS3} 
some more classes of manifolds of codimension exceeding one are provided (the concrete examples found have codimension two or three), on the other hand, this condition induces 
some rigid restrictions on the dimension pairs $(n,d)$. Since $\textbf{g}$ is analytic, presumably
condition (IIp) can be simplified. Indeed,
then $\det M=\det M_{\zz}(\xx)$ as a function in $x_1,\ldots,x_d$ is a $\Q_p$-valued analytic map for any $\zz\in \Z^{n-d}$ as well. We expect that just like the real case (by Lojasiewicz's stratification theorem, see for example~\cite{loj}), the exceptional set $S_{\MM}(\zz)$ can have 
$p$-adic dimension at most $d-1$, unless $\det M_{\zz}$ is identically $0$. We have not found a proper reference for the $p$-adic version.
If true, this would imply that in (IIp) if
$f$ satisfies
\begin{equation}  \label{decfast}
    f(r) < r^{d-1+\varepsilon},\qquad \varepsilon>0,\; r\in (0,r_0),
\end{equation}
then we only need to assume that $\det M_{\zz}(\xx)$ is not the constant $0$ map for any non-zero rational integer vector $\zz$. 
Condition (Ip) for some $s>d-1$ may
be sufficient to guarantee \eqref{decfast}.

Consider the subsets of $\Z^{n+1}$ defined by
\begin{align*}
Z(1)&=\left\{ \ba=(a_{0}, \dots , a_{n}) \in \Z\times(\Z^{n}\backslash\{\0\}): \rm{gcd}(a_{i},a_{j},p)=1, \quad 0\leq i < j \leq n \right\}, \\
Z(2)&=\left\{ \ba=(a_{0}, \dots , a_{n}) \in \Z\times(\Z^{n}\backslash\{\0\}): \rm{gcd}(a_{0},\dots ,a_{n},p)=1 \right\}.
\end{align*}
That is the sets where $p$ divides at most one $a_i$ and not all $a_i$, respectively.
Notice that
\[
Z(1)\subseteq Z(2)\subseteq \Z \times (\Z^{n}\backslash\{\0\})\, .
\]

The reason that $Z(1)\subseteq Z(2)\subseteq \Z\times(\Z^{n}\backslash\{\0\})$ rather than $Z(1)\subseteq Z(2) \subseteq \Z^{n+1}\backslash\{\0\}$ is to ensure we are not just approximating the inhomogenity $\theta$ by elements of $\Z$, which is not our focus in this paper.  

We prove the following. 

\begin{theorem} \label{dual_hypersurface_inhomogeneous}
Let $f$ be a dimension function satisfying (Ip) and $\MM=\{(\xx,\textbf{g}(\xx)): \xx \in \Qp^{d}\}$ be a manifold of dimension $d$ defined by analytic $\textbf{g}:\UU \subset \Qp^{d} \to \Qp^{n-d}$ satisfying (IIp). Let $\theta \in \Zp$ and $\Psi:\Z^{n+1} \to \R_{+}$ be a multivariable approximation function with $\Psi(\ba)<\|\ba\|^{-1}$. Then
\begin{enumerate}
\item[\rm (i)] $\HH^{f}(D_{n,p}^{\theta}(\Psi,Z(1)) \cap \MM) =0$ if
\begin{equation*}
\sum_{\ba \in Z(1)}\Psi(\ba)^{-(d-1)}f(\Psi(\ba))<\infty. 
\end{equation*}
\item[\rm (ii)]  $\HH^{f}(D_{n,p}^{\theta}(\Psi,Z(2)) \cap \MM) =0$ if $|\theta|_p\neq1$ and
\begin{equation*}
\sum_{\ba \in Z(2)}\Psi(\ba)^{-(d-1)}f(\Psi(\ba))<\infty\,.
\end{equation*}
\item[\rm (iii)]  $\HH^{f}(D_{n,p}(\Psi) \cap \MM) =0$ if
\begin{equation*}
\sum_{\ba \in \Z^{n+1}\backslash\{ \0\}}\Psi(\ba)^{-(d-1)}f(\Psi(\ba))<\infty \, , \end{equation*} 
and
 \begin{equation} \label{lasten}
 p\Psi(\ba)\le \Psi(p^{-1}\ba) \qquad \, \, \forall \, \, \ba \in p\Z^{n+1}\backslash \{ \0\}.  
\end{equation} 
\end{enumerate}  
\end{theorem}  


The limitation $\Psi(\ba)<\|\ba\|^{-1}$ is not too restrictive. Indeed, if $\Psi(\ba)>\|\ba\|^{-1}> \|\ba\|^{-n}$ for all sufficiently large $\|\ba\|$, then by Dirichlet's Theorem $D_{n,p}(\Psi)=\Qp^{n}$.
According to (i), restricting the integer vectors
to the smallest set $Z(1)$,
we have the perfect convergence claim as predicted by the GBSP. In the classical homogeneous setting,
the same is true when summing over $Z(2)$, by (ii).
The additional condition on $\theta$ in (ii) is an artifact of our method, we do not have a satisfactory explanation for its occurrence.
Claim (iii) is for the homogeneous setting only.
Note that \eqref{lasten} is 
clearly satisfied for the standard power functions
$\Psi(\ba)=\Vert \ba\Vert^{-\tau}, \tau\ge n$,
and more generally for any function $\Psi$ 
with the property that
$-\log \Psi(\ba)/\Vert \ba\Vert$ is
non-decreasing. \par

\subsection{Corollaries to Theorem~\ref{dual_hypersurface_inhomogeneous} combined with results of Datta and Ghosh } \label{DatGho}

The following result is a special case of more general
results from~\cite{DattaGhosh}, for $s$-dimensional Hausdorff measure.

\begin{theorem}[Datta and Ghosh 2022]  \label{dago}
Suppose that $\textbf{g}:\UU \subset \Qp^{d}\to\Qp^{n}$ and satisfies
\begin{enumerate}
\item[\rm (I)] $\textbf{g}$ is an analytic map and can be extended to the boundary of $\UU\subset\Qp^{d}$ an open ball,
\item[\rm (II)] Assume that $1,x_{1}, \dots , x_{d}, g_{1}(\xx), \dots, g_{n-d}(\xx)$ are linearly independent functions over $\Qp$ on any open subset of $\UU\ni \xx$,
\item[\rm (III)] Assume that $\|\textbf{g}(\xx)\|_{p}\leq 1$ and $\|\nabla \textbf{g}(\xx)\|_{p}\leq 1$ for any $\xx\in \UU$, and that the second order difference quotient\footnote{ See \cite[Section 5]{DattaGhosh} or \cite{KT07} for the definition of second order difference quotients.} is bounded from above by $\frac{1}{2}$ for any multi-index and any triplets $\yy_{1},\yy_{2},\yy_{3} \in \UU$.
\end{enumerate}
Let
\begin{equation*}
\Phi(\ba)=\phi(\|\ba\|), \quad \forall \ba \in \Z^{n+1}
\end{equation*}
for $\phi:\N \to \R_{+}$ a positive non-increasing function, and assume that $s>d-1$. Then
\begin{equation*}
\HH^{s}(D_{n,p}(\Phi)\cap \MM)=\HH^{s}(\MM) \quad \text{ if } \quad \sum_{\ba\in\Z^{n+1}\backslash\{0\}}
\Phi(\ba)^{s+1-d}=\infty.
\end{equation*}
\end{theorem}

\begin{remark} \rm
Note that in \cite{DattaGhosh} the set
\begin{equation*}
W^{\textbf{g}}_{\Phi,\theta}=\{\xx \in \UU: (\xx,\textbf{g}(\xx)) \in D_{n,p}^{\theta}(\Phi) \}
\end{equation*}
is considered. Since $\nabla \textbf{g}$ is bounded on $\UU$ there exists a bi-Lipschitz map between the two sets $W^{\textbf{g}}_{\Phi,\theta}$ and $D_{n,p}^{\theta}(\Phi)\cap \MM$ (or rather their complements) and so full measure results are equivalent. See the start of Section~\ref{complete} for further details of such equivalence.
\end{remark}

\begin{remark}   \label{dattago}
	The divergence case in the inhomogeneous setting
	is rather general. In particular, they prove the result for some general analytic function $\mathbf{\Theta}:\UU \to \Zp$ satisfying certain conditions, see \cite[Condition (I5)]{DattaGhosh} for more details. In Theorem~\ref{dago} and our application below, since we relate to~Theorem~\ref{dual_hypersurface_inhomogeneous} (iii), we consider the homogeneous setting $\mathbf{\Theta}(\xx)=0$.
\end{remark}

Combining our convergence result with Theorem~\ref{dago}, we have the following statement for the homogeneous case and $s$-dimensional Hausdorff measure.

\begin{theorem} \label{cor1}
Let $f(r)=r^{s}$ be a dimension function with
\begin{equation*}
d-1<s<2(d-1).
\end{equation*}
Let $\textbf{g}:\UU\subset \Qp^{d} \to \Qp^{n}$ be an analytic map satisfying (IIp), (I), (II), and (III). Let
\begin{equation}  \label{psi-form}
\Psi(\ba)=\frac{\psi(\|\ba\|)}{\|\ba\|}
\end{equation}
for a monotonic decreasing function $\psi:\N\to\R_{+}$ tending to zero.
Then
\begin{equation*}
\HH^{s}(D_{n,p}(\Psi)\cap \MM)=\begin{cases}
0 & \text{ if }\quad \sum\limits_{\ba \in \Z^{n+1}} \Psi(\|\ba\|)^{s+1-d}<\infty, \\[2ex]
\HH^{s}(\MM)& \text{ if }\quad  \sum\limits_{\ba \in \Z^{n+1}} \Psi(\|\ba\|)^{s+1-d}=\infty.
\end{cases}
\end{equation*}
\end{theorem}

\begin{proof}
The lower bound on $s$ is due to Theorem~\ref{dago}, and the upper bound is due to (iii) in Theorem~\ref{dual_hypersurface_inhomogeneous}. The conditions on $\textbf{g}$ are a combination of requirements from both theorems. 
We identify
$\phi(\Vert\ba\Vert)= \psi(\Vert\ba\Vert)/\Vert \ba\Vert$ so that $\Phi(\ba)= \Psi(\ba)$, noting that $\phi$ is clearly decreasing as well.
Thus on the one hand we may apply Theorem~\ref{dago}, on the other hand, we may 
deduce that
\begin{equation*}
\Psi(\ba)=\frac{\psi(\|\ba\|)}{\|\ba\|} =\frac{\psi(\|\ba\|)}{p\|p^{-1}\ba\|}\leq \frac{\psi(\|p^{-1}\ba\|)}{p\|p^{-1}\ba\|}=p^{-1}\Psi(p^{-1}\ba),
\end{equation*}
and so Theorem~\ref{dual_hypersurface_inhomogeneous} (iii) is applicable.
\end{proof}

Since $\HH^{s}(\MM)=0$
for $s>d$, if $d\ge 2$ the theorem covers the whole interesting range $(d-1,d]$ for $s$, apart from the endpoint $s=d=2$ when $d=2$.
Since \eqref{decfast} is satisfied for $f$
in Theorem~\ref{cor1}, we expect that (IIp) can be relaxed to impose that $\det M(\xx)=\det M_{\zz}(\xx)$ is not constant $0$, for any choice of $\zz$.

\begin{remark}\rm
By the above remarks on (Ip), (IIp) in Section~\ref{padon},
the above result gives us a zero-full dichotomy for homogeneous dual approximation on sufficiently curved hypersurfaces of $n$-dimensional $p$-adic space with $n\geq 3$. 
Note that individually each result (the convergence and divergence case) extends beyond the above theorem. For example, \cite[Theorem 1.1]{DattaGhosh} allows for the inhomogeneous setting, but does not allow for multivariable approximation. Additionally, we have restricted the range of admissible approximation functions to be those of the form \eqref{psi-form}. Note that such a form of approximation function is permissible in applying Theorems \ref{dual_hypersurface_inhomogeneous}-1.5.
\end{remark}

Denoting by $\dim_\HH$ the Hausdorff dimension,
we deduce the following immediate corollary.

\begin{corollary}\label{cor2}
Let $\textbf{g}$ be as in Theorem~\ref{cor1} and $d\geq 2$. Suppose that
\begin{equation*}
\Psi(\ba)=\|\ba\|^{-1-\tau}
\end{equation*}
for some $\tau> n$. Then
\begin{equation*}
\dim_\HH(D_{n,p}(\Psi)\cap \MM) = d -1 + \frac{n+1}{1+\tau}.
\end{equation*}
\end{corollary}

\begin{proof} 
We first check that all the requirements of Theorem~\ref{cor1} are satisfied. Note that $s=d-1+
\frac{n+1}{\tau+1}<2(d-1)$ whenever $\tau>\frac{n+1}{d-1}-1$, and that $n\geq \frac{n+1}{d-1}-1$ for all $d\geq 2$,
hence our assumption $\tau>n$ implies $s<2(d-1)$. The remaining 
conditions are obvious.
So it remains to find when the critical sum converges/diverges. Note that
\begin{equation*}
\sum_{\ba\in\Z^{n+1}}\|\ba\|^{-(1+\tau)(s+1-d)}= \sum_{r=1}^{\infty}\sum_{\ba \in \Z^{n+1}:\; \\ \|\ba\|=r}r^{-(1+\tau)(s+1-d)}\asymp\sum_{r=1}^{\infty}r^{n-(1+\tau)(s+1-d)}
\end{equation*}
and the summation on the right hand side converges for $s>d-1+\frac{n+1}{1+\tau}$ and diverges when $s\leq d-1+ \frac{n+1}{1+\tau}$. 
\end{proof}

\begin{remark} \rm
Note trivially for $\tau\leq n$ we have that $D_{n,p}(\Psi)=\Qp^{n}$ by the $p$-adic version of Dirichlet's Theorem, and so $\dim_\HH(D_{n,p}(\Psi)\cap \MM)=\dim_\HH \MM =d$.
\end{remark}

\begin{remark}
    Note that in contrast to~\cite{HSS1} where $d
\ge 3$ is required, we get
    a claim for $d=2$ here. The reason is that
    we assume strict inequality $\tau>n$, so
    that the parameter range for $s$ satisfying
    (Ip) can be improved. Any such improvement
    leads to the implementation of the case $d=2$.
\end{remark}

\noindent{\bf Acknowledgments:} The research of Mumtaz Hussain and Ben Ward is supported by the Australian Research Council discovery project 200100994. Part of this work was done when Johannes visited La Trobe University, we thank the Sydney Mathematics Research Institute and La Trobe University for the financial support.

\section{ Preliminaries and the main lemma}

\subsection{Preliminaries} \label{prelim}

We first note that for self-mappings (that is $m=n$) of analytic functions as in the Definition~\ref{Defana}, the inverse function theorem holds \cite[Section 9 p. 113]{Serre}. The claim is formulated more generally for any ultrametric space.
We notice that as explained in \cite{Serre} the notion of analyticity implies infinite differentiability in a strong sense. Hereby we call a function $\phi: U\subseteq \mathbb{Q}_{p}^{m}\to \mathbb{Q}_{p}^{n}$ 
strong differentiable at $\mbf{a}$ if there is a linear function $L: \mathbb{Q}_{p}^{m}\to \mathbb{Q}_{p}^{n}$ such that
\begin{equation} \label{eq:new}
\lim_{|\hh|_{p}\to 0, \hh\neq 0} \frac{\|\phi(\mbf{a}+\hh)-\phi(\mbf{a})-L\hh\|_{p}}{\|\hh\|_{p}}=0. 
\end{equation}
In the case $m=n=1$, this is equivalent to the existence of the limit
\[
\lim_{(x,y)\to (a,a), x\neq y} \frac{|\phi(x)-\phi(y)|_{p}}{|x-y|_{p}}.
\]
Such a derivative is uniquely determined (if it exists) and
we denote it $L\mbf{a}=D\phi(\mbf{a})$. Notice
that in the $p$-adic setting this is indeed stronger than the typical
notion of differentiability where $L\mbf{a}$ is involved instead of $L\hh$
in the numerator of \eqref{eq:new}, 
and one has to be careful about transferring
real/complex analysis claims to the $p$-adic world. 
See \cite[Example 26.6]{Schikhof} for a famous example of a function
$f: \mathbb{Z}_{p}\to \mathbb{Q}_{p}$
with ``ordinary'' derivative $Df\equiv 1$ which is not injective in any neighborhood of $0$. 
However, with the above notion of strong differentiability most 
common real analysis facts can be conserved. 
If we assume our function to be analytic then the inverse function theorem 
extends to the $p$-adic settings as claimed in Serre's book \cite{Serre}. 

\begin{theorem}[$p$-adic Inverse Function Theorem]  \label{ift}
	Assume the function $\phi: U\subseteq \mathbb{Q}_{p}^{n}\to \mathbb{Q}_{p}^{n}$ 
	for $\0\in U$ open is analytic. Then if $D\phi(\0)$ induces a linear isomorphism,
	then $\phi$ is a local isomorphism. 
	\end{theorem}
We remark that when $n=1$, in fact, the function
$\phi/\phi^{\prime}(0)$ is a local isometry,
i.e. $|\phi(x)-\phi(y)|_{p}/|x-y|_{p}= |D\phi(a)|_{p}$ is constant
for $x,y$, $x\ne y$, close enough to $a$, 
see \cite[Proposition 27.3]{Schikhof}. We further point out
that Theorem~\ref{ift} is the only reason why we require 
our parametrising function $\textbf{g}$ to be analytic, for all other arguments $C^{2}$ would suffice.

From the inverse function theorem it can be shown that some
mean value estimate similar to the real case holds.

\begin{theorem}[$p$-adic Mean Value Theorem] \label{pmvt}
	For $\phi$ as in Theorem~\ref{ift} and small enough $r$ we have
	\[
	\phi(B(\xx, r)) \subseteq B(\phi(\xx), \Vert D\phi(\xx)\Vert_{p}\cdot r).
	\]
\end{theorem}

In other words a ball of radius $r$ (with respect to $\Vert .\Vert_{p}$-norm) around $\xx\in\Q_{p}^{n}$ 
is mapped into a ball of radius $\ll_{\xx} r$.  In fact, the mean
value estimate only requires
strong differentiability.

\subsection{The main lemma}

Let $\UU \subset\Qp^{d}$ be a connected bounded open set and
\begin{equation*}
\Lambda_{\MM}:= \underset{\|\zz\|_{p}=1}{\bigcup_{\zz\in \Z^{n-d}:}} S_{\MM}(\zz)=\underset{\|\zz\|_{p}=1}{\bigcup_{\zz\in \Z^{n-d}:}}\left\{ \xx \in \UU : M_{\zz}(\xx)=\left( \sum_{k=1}^{n-d}z_{k}\frac{\partial^{2} g_{k}(\xx)}{\partial x_{i} \partial x_{j}} \right)_{1\leq i, j \leq d} \, \text{ is singular} \right\} .
\end{equation*}
Observe that for each individual vector $\zz \in \Z^{n-d}\backslash\{\0\}$ 
of $p$-norm $1$, the above set $S_{\MM}(\zz)$ has $\HH^{f}$-measure zero by condition (IIp), and so by the sigma-additivity of the $\HH^{f}$ measure we have that
\begin{equation*}
\HH^{f}(\Lambda_{\MM})=0.
\end{equation*}
Thus we may restrict to fixed $\zz$ and just write $S_{\MM}=S_{\MM}(\zz)$.
Observe that $\UU \backslash S_{\MM}$ can be covered by countably many small open balls $B_{i}$ such that $\det M(\xx)$ is bounded on $B_{i}$ and $\det M(\xx) >\varepsilon_{i}$ for some $\varepsilon_{i}>0$ and all $\xx \in B_{i}$. 
 If we can show that
\begin{equation*}
\HH^{f}\left( D_{n,p}^{\theta}(\Psi) \cap \{(\xx,\textbf{g}(\xx)): \xx \in B_{i}\} \right) =0
\end{equation*}
for all $B_{i}$, then by the subadditivity of $\HH^{f}$ we would have that
\begin{equation*}
\HH^{f}\left( D_{n,p}^{\theta}(\Psi) \cap \MM \right) =0.
\end{equation*}
Hence without loss of generality let us assume that $\UU$ is an open ball with the $d\times d$ matrices $M(\xx)$  with determinant bounded away from $\0$. Furthermore suppose that $\UU \subset \Zp^{d}$ and $\0\in \UU$. \par

For any $\ba=(a_{0},a_{1}, \dots , a_{n}) \in \Z^{n+1}$ consider the set
\begin{equation*}
S(\ba):=\left\{ \xx \in \UU : |a_{1}x_{1}+\dots a_{d}x_{d} + a_{d+1}g_{1}(\xx)+ \dots + a_{n}g_{n-d}(\xx) +a_{0} + \theta |_{p}< \Psi(\ba) \right\}.
\end{equation*}
We prove the following key lemma.

\begin{lemma} \label{S(a)_haus_cont}
 For any $\ba \in \Z \times (\Z^{n}\backslash \{\0\})$, $\textbf{g}:\UU \to \Qp^{n}$ satisfying (IIp), and dimension function $f$ satisfying (Ip) for some $s<2(d-1)$, we have that
\begin{equation*}
\HH^{f}_{\infty}(S(\ba)) \ll  \|(a_{1}, \dots ,a_{n})\|_{p}^{(d-1)-s}\Psi(\ba)^{-(d-1)} f\left( \Psi(\ba)\right),
\end{equation*}
with the implied constant independent of $\ba$.
\end{lemma}

\section{Proof of Lemma~\ref{S(a)_haus_cont}}

We may restrict $K$ to a compact subset of $\UU$. 
For $\ba \in \Z^{n+1}$ write 
\begin{equation*}
\ba=(a_{0}, \ba_1,\ba_2)=(a_{0},a_{1},\dots,a_{d}, a_{d+1}\dots,a_{n}) \in \Z \times \Z^{d}\times \Z^{n-d},
\end{equation*}
 and let 
\begin{equation*}
\widetilde{a_{2}}=\begin{cases}
\,\,\, \, 1 \quad \quad  \,\, \text{ if } \ba_{2}=\0, \\
\|\ba_{2}\|_{p}\quad \text{ otherwise.}
\end{cases}
\end{equation*}
 Define $h_{\ba}:\Qp^{d} \to \Qp$ by
\begin{equation*}
h_{\ba}(\xx)= \widetilde{a_{2}}\ba_{1}\cdot \xx + \widetilde{a_{2}}\ba_{2}\cdot \textbf{g}(\xx) + \widetilde{a_{2}}(a_{0}+\theta) \, ,
\end{equation*}
where
\begin{equation*}
\widetilde{a_{2}}\ba_{1}\cdot\xx=\widetilde{a_{2}}a_{1}x_{1}+\dots + \widetilde{a_{2}}a_{d}x_{d},
\end{equation*}
and similarly for $\widetilde{a_{2}}\ba_{2}\cdot \textbf{g}(\xx)$. We may write
\begin{equation*}
S(\ba)=\left\{ \xx \in K : |h_{\ba}(\xx)|_{p}< \widetilde{a_{2}}^{-1}\Psi(\ba) \right\} \, .
\end{equation*}
Note that
\begin{align*}
\nabla h_{\ba}(\xx)&=\left( \frac{\partial h_{\ba}(\xx)}{\partial x_{1}}, \dots , \frac{\partial h_{\ba}(\xx)}{\partial x_{d}} \right), \\
&=\left(\widetilde{a_{2}}a_{1}+\sum_{k=1}^{n-d}\widetilde{a_{2}}a_{d+k}\frac{\partial g_{k}(\xx)}{\partial x_{1}} \, , \, \dots \, , \, \widetilde{a_{2}}a_{d}+\sum_{k=1}^{n-d}\widetilde{a_{2}}a_{d+k}\frac{\partial g_{k}(\xx)}{\partial x_{d}} \right) , \\
&=\widetilde{a_{2}}\ba_{1} +\widetilde{a_{2}}\ba_{2} \cdot \left(\nabla \textbf{g}(\xx)\right) ,
\end{align*}
for 
\begin{equation*}
\nabla \textbf{g}(\xx)=\left(\frac{\partial \textbf{g}(\xx)}{\partial x_{1}}, \dots , \frac{\partial \textbf{g}(\xx)}{\partial x_{d}}\right) \quad \text{ with } \quad \frac{\partial \textbf{g}(\xx)}{\partial x_{i}}=\left( \frac{\partial g_{1}(\xx)}{\partial x_{i}}, \dots , \frac{\partial g_{n-d}(\xx)}{\partial x_{i}} \right),
\end{equation*}
a matrix with $n-d$ rows and $d$ columns and
\begin{equation*}
\nabla^{2} h_{\ba}(\xx)=\left( \sum_{k=1}^{n-d}\widetilde{a_{2}}a_{d+k}\frac{\partial^{2} g_{k}(\xx)}{\partial x_{i} \partial x_{j}} \right)_{1\leq i,j \leq d}.
\end{equation*}
Note that if $\ba_{2}\neq \0$ the vector $\left\|\widetilde{a_{2}}\ba_{2}\right\|_{p}=1$ by our choice of $\widetilde{a_{2}}$, and so 
identifying $z_k=\widetilde{a_{2}} a_{d+k}$ for 
$k=1,2,\ldots,n-d$, by our assumption (IIp)
$\nabla^{2}h_{\ba}(\xx)$ has a non-zero determinant outside a set of $\HH^f$-measure $0$. We may, for simplicity, assume that the 
exceptional set is empty, $K\cap \Lambda_{\MM}=\emptyset$.
 We next bound $\nabla^2 h_{\ba}$ uniformly from above. \\
 
\begin{lemma} \label{h_lem} For $\ba \in \Z^{n+1}$ and $h_{\ba}$ defined as above
\begin{enumerate}
\item[\rm (a)] If $\ba_{2}=\0$ then 
$$
\|\nabla h_{\ba}(\xx)\|_{p} \asymp \|\ba_{1}\|_{p}  \quad \forall \, \xx \in K.
$$
\item[\rm (b)] If $\ba_{2}\ne \0$ and $\|\ba_{1}\|_{p}>\sup_{\ww \in K}\left\|\ba_{2}\cdot \nabla \textbf{g}(\ww)\right\|_{p}$ then 
$$
\|\nabla h_{\ba}(\xx)\|_{p} \asymp \|\ba_{2}\|_{p}^{-1}\|\ba_{1}\|_{p}  \quad \forall \, \xx \in K.
$$
\item[\rm (c)] If there exists $\vv \in K$ with $\nabla h_{\ba}(\vv)=\0$ then 
$$
\|\nabla h_{\ba}(\xx)\|_{p} \asymp \|\xx-\vv\|_{p}  \quad \forall \, \xx \in K.
$$
\item[\rm (d)] If no such $\vv \in K$ exists, and  $\|\ba_{1}\|_{p}\leq \sup_{\ww \in K}\|\ba_{2}\cdot \nabla \textbf{g}(\ww)\|_{p}$ then 
$$
0<\|\nabla h_{\ba}(\xx)\|_{p} \leq \sup_{\ww \in  K}\|\nabla \textbf{g}(\ww)\|_{p} \quad \implies \quad \|\nabla h_{\ba}(\xx)\|\asymp 1 \quad \forall \, \xx \in K.
$$
\end{enumerate}
In all of the above cases $\|\nabla^{2}h_{\ba}(\xx)\|_{p} \ll 1$ for all $\xx\in K$.
\end{lemma}

\textit{Proof of Lemma~\ref{h_lem}:} Note that if $\ba_{2}=\0$ immediately we have that
\begin{equation*}
\| \nabla h_{\ba}(\xx)\|_{p}=\|\ba_{1}\|_{p}
\end{equation*}
for all $\xx \in K$, giving (a). \par 
 If  $\|\widetilde{a_{2}}\ba_{1}\|_{p}>\sup_{\ww \in K}\left\| \widetilde{a_{2}}\ba_{2}\cdot \nabla \textbf{g}(\ww)\right\|_{p}$, which follows by the assumption
 \[
 \|\ba_{1}\|_{p}>\sup_{\ww \in K}\left\|\ba_{2}\cdot \nabla \textbf{g}(\ww)\right\|_{p},
 \]
 then by the strong triangle inequality
\begin{equation*}
\|\nabla h_{\ba}(\xx)\|_{p}= \max\left\{\|\widetilde{a_{2}}\ba_{1}\|_{p}, \left\|\widetilde{a_{2}}\ba_{2}\cdot\nabla \textbf{g}(\xx)\right\|_{p} \right\}=\|\widetilde{a_{2}}\ba_{1}\|_{p}.
\end{equation*}
Henceforth assume $\|\widetilde{a_{2}}\ba_{1}\|_{p}\leq \sup_{\ww \in K}\left\| \widetilde{a_{2}}\ba_{2}\cdot \nabla \textbf{g}(\ww)\right\|_{p}$. \\

Since $\textbf{g}$ is analytic, and thus so is each $g_{i}$ for $1\leq i \leq n-d$, each partial derivative $\frac{\partial g_{i}}{\partial x_{j}}$ is analytic and so the function $\widetilde{a_{2}}\ba_{2}\cdot \nabla \textbf{g}: K\subset\Qp^{d} \to \Qp^{d}$ is analytic. 
Furthermore $\widetilde{a_{2}}\ba_{2}\cdot\nabla \textbf{g}$ is strongly differentiable with the linear function $L\xx =\nabla^{2}h_{\ba}(\xx)$ and non-zero determinant (thus invertible) on some small ball $B_{\xx}\subset K$. Hence $\nabla^{2}h_{\ba}(\xx)$ is a linear isomorphism and so $\widetilde{a_{2}}\ba_{2}\cdot\nabla \textbf{g}$ is a local isomorphism by Theorem~\ref{ift}. Hence by Theorem~\ref{pmvt} for any $\yy \in B_{\xx}$ 
\begin{equation} \label{needit}
\|\nabla h_{\ba}(\xx)-\nabla h_{\ba}(\yy)\|_{p}= \| \widetilde{a_{2}}\ba_{2}\cdot \nabla g(\xx)-\widetilde{a_{2}}\ba_{2}\cdot\nabla g(\yy)  \|_{p} \asymp \| \xx-\yy\|_{p}.
\end{equation}
The above implied constant depends on $\widetilde{a_{2}}\ba_{2}$. However, $\widetilde{a_{2}}\ba_{2} \in \{\zz \in \Z^{n-d} : \|\zz\|_{p}=1\}$ is a compact set.
Indeed, we have
\begin{equation}  \label{eq:Wnew}
\Vert \widetilde{a_{2}}\ba_{2}\Vert_p = \vert \widetilde{a_{2}}\vert_p \cdot \Vert\ba_{2}\Vert_p= \vert \;\Vert \ba_{2}\Vert_p\; \vert_p\cdot \Vert\ba_{2}\Vert_p= 
\Vert \ba_{2}\Vert_p^{-1}\cdot \Vert\ba_{2}\Vert_p = 1,
\end{equation}
where we employed the identity $\vert\; \vert y\vert_p\; \vert_p= \vert y\vert_p^{-1}$ for any $y\in \Q_p$ that is readily checked.
So uniform bounds exist. This settles (b).
 \par
Assume as in (c) there exists some $\vv \in K$ such that $\nabla h_{\ba}(\vv)=0$ (or $\|\nabla h_{\ba}(\vv)\|_{p}< \| \nabla h_{\ba}(\xx)\|_{p}$ for all $\xx \in K\backslash \{\vv\}$)\footnote{Observe that this case is not possible. To see this note that $\nabla h_{\ba}$ is continuous, and so if $\nabla h_{\ba}(\vv) \neq 0$ there exists a neighbourhood $V$ of $\vv$ such that $\|\nabla h_{\ba}(\vv)\|_{p}=\|\nabla h_{\ba}(\xx)\|_{p}$ for all $\xx \in V$. See \cite[Section 3.2]{KT07}}. Then we can apply \eqref{needit} identifying $\vv=\yy$ to get in a neighborhood
$B_{\vv}$ of $\vv$ that
\begin{equation*}
\| \nabla h_{\ba}(\xx) \|_{p}=\| \nabla h_{\ba}(\xx) - \nabla h_{\ba}(\vv) \|_{p}  \asymp \|\xx - \vv \|_{p}, \qquad \xx\in B_{\vv}.
\end{equation*}
So the claim holds within $B_{\vv}$. In the complement of any such $B_{\vv}$ within $K$ (in fact we may assume $\vv$ is unique in $K$), we can  
consider ourselves to be in case (d).
 So suppose as in (d) we have $\nabla h_{\ba}(\xx) \neq \0$ for all $\xx \in K$, and so $\| \nabla h_{\ba}(\xx)\|_{p}\gg 1$ uniformly by compactness of $K$ and the continuity of $\nabla h_{\ba}$. By the strong triangle inequality, we have that
\begin{equation} \label{h_bounds}
\| \nabla h_{\ba}(\xx)\|_{p} \leq \max \left\{ \|\widetilde{a_{2}}\ba_{1}\|_{p} , \|\widetilde{a_{2}}\ba_{2} \cdot \nabla \textbf{g}(\xx) \|_{p} \right\}.
\end{equation}
By our assumption
\begin{equation*}
 0< \|\widetilde{a_{2}}\ba_{1}\|_{p}  \leq \sup_{\xx \in K} \|\widetilde{a_{2}}\ba_{2} \cdot \nabla \textbf{g}(\xx) \|_{p} \leq \sup_{\xx \in K} \max_{1\leq i\leq d}\left\|\frac{\partial \textbf{g}(\xx)}{\partial x_{i}} \right\|_{p} \ll 1
\end{equation*}
and so by \eqref{h_bounds} we have that
\begin{equation*}
\|\nabla h_{\ba}(\xx) \|_{p} \ll1.
\end{equation*}
Combined with $\| \nabla h_{\ba}(\xx)\|_{p}\gg 1$ we have 
\begin{equation*}
  \| \nabla h_{\ba}(\xx)\|_{p} \asymp 1 .
\end{equation*}
The final claim follows from the compactness of $K$
and the continuity of $\nabla^2 h_{\ba}=\nabla^2 \textbf{g}$.
 \qed

We prove the following key lemma which is a $p$-adic analogue of \cite[Lemma 2.4]{HSS1}.

\begin{lemma} \label{ll}
	Let $\phi: U \subset \Q_{p}^{d}\to \Q_{p}$ be an analytic
	function, and fix $\alpha>0$, $\delta>0$, and $\xx\in U$ such that $B_{d}(\xx,\alpha) \subset U$.
	There exists a constant $C > 0$ depending only on $d$ such that if
	\[
	\|\nabla \phi(\xx)\|_{p} \geq C \alpha \sup_{\ww\in U} 
	\|\nabla^2 \phi(\ww)\|_{p},
	\]
	then $$S(\phi,\delta) = \{\yy\in B_{d}(\xx,\alpha) : |\phi(\yy)|_{p} < \|\nabla \phi(\xx)\|_{p} \delta\}$$ can be covered by $\asymp (\alpha/\delta)^{d-1}$ balls in $\mathbb{Q}_{p}^{d}$ of radius $\delta$.
\end{lemma} 

Our proof will slightly deviate from the real case in \cite{HSS1}, but 
we employ the same idea.

\begin{proposition} \label{ppp1}
	Let $d\in\N$. The invertible matrices form an open subset of $\Q_{p}^{d\times d}$.
\end{proposition}

\begin{proof}
	Let $A$ be an invertible matrix. Expand the determinant $A_{\epsilon}=A+\epsilon Y$ with
	any fixed matrix $Y \in \Q_{p}^{d\times d}$ with $\det Y=1$, and $\epsilon\in \Zp$. This gives $$\det A_{\epsilon}=\det A+\epsilon z_{1}+\epsilon^{2}z_{2}+\cdots+\epsilon^{d}z_{d}$$ for some fixed $z_{i}\in\Q_{p}$ built from entries of $A$ and $Y$. Since $\det A\neq 0$ if we let
	$|\epsilon|_{p}<1$ small enough then $$\Vert \det A_{\epsilon}\Vert_{p}\geq \Vert \det A\Vert_{p}- |\epsilon|_{p} \max |z_{i}|_{p}$$ will be non-zero. \\
\end{proof}

\begin{proposition} \label{ppp2}
	Let $d \in \N$. If $\Phi: \Q_{p}^{d}\to \Q_{p}^{d}$ is a Lipschitz 
	map with Lipschitz constant $L$, 
	and $U\subseteq \Q_{p}^{d}$ can be covered
	by $k$ balls of radius $r>0$, then $\Phi(U)$ can be covered by $\ll_{d,L} k$ balls of the same radius $r$.
\end{proposition}

\begin{proof}
Let $B(\xx,r)$ be one of the $k$ balls that cover $U$. For any $\yy \in B(\xx,r)$ since $\Phi$ is Lipschitz we have
\begin{equation*}
\| \Phi(\xx)-\Phi(\yy) \|_{p}\leq L \|\xx-\yy\|_{p} \leq Lr,
\end{equation*}
and so $\Phi(B(\xx,r)) \subseteq B(\Phi(\xx), Lr)$ and so $\Phi(U)$ can be covered by $k$ balls of radius $Lr$. We now show the following claim:

\textit{
 For any $\xx \in \Qp^{d}$ and any $\rho, K>0$ the $p$-adic ball $B_{d}(\yy,K\rho)$ can
 be covered by $\ll \max\{1,K^{d}\}$ balls of radius $\rho>0$.}

Using this we can deduce that $\Phi(U)$ can be covered by $\ll L^{d}k$ balls as required. To prove the statement recall by the properties of the ultrametric norm that any two balls are either disjoint or one contains the other, that is,  the intersection is either empty or full. If $1\geq K>0$ then trivially $B(\yy,r)$ is a cover of $B(\yy,K\rho)$ so the number of balls is $\ll 1$. So assume $K>1$.  Let $\{B_{i}\}$ be a collection of balls of radius $\rho$ that cover $B(\yy,K\rho)$.  We can assume this collection is disjoint and finite by the properties of the ultrametric norm and since $B(\yy,K\rho)$ is bounded. Furthermore we can assume 
\begin{equation*}
B(\yy,K\rho)=\bigcup_{i}B_{i}.
\end{equation*}
Now by properties of the $p$-adic $d$-dimensional Haar measure $\mu_{p,d}$ we have that
\begin{equation*}
(K\rho)^{d} \asymp \mu_{p,d}\left( B(\yy,K\rho) \right) = \sum_{i} \mu_{p,d}\left( B_{i} \right) \asymp \sum_{i} \rho^{d}
\end{equation*}
Hence the cardinality of the set of balls $\{B_{i}\}$ is $\asymp K^{d}$.\\
\end{proof}

\begin{proof}[Proof of Lemma~\ref{ll}]
	We proceed as in the proof of \cite[Lemma 2.4]{HSS1}. By translation, without loss of generality, we may assume $\xx=\0$.
	For simplicity let $\kappa:=\|\nabla \phi(\0)\|_{p}$. Clearly $\kappa>0$ as otherwise
	by the assumption of the lemma $\nabla^{2}\phi$ vanishes on an open set which we can easily exclude.
	By rotation we may assume $\nabla \phi(\0)=\kappa\ee_{d}$, where $\ee_{d}$ denotes the $d$-th canonical base vector
	in $\Q_{p}^{d}$. Then $\Vert \nabla \phi(\0)\Vert_{p}=\kappa$.
	Now consider the map 
	$$\Phi: \mathscr{B} := B_{d}(\0,\alpha)\to \Q_{p}^{d}$$ 
	defined by the formula $$\Phi(\yy) = (\kappa y_1,\ldots,\kappa y_{d-2},\kappa y_{d-1},\phi(\yy)),$$
	for $\yy=(y_{1},\ldots,y_{d})\in \Q_{p}^{d}$. We have $$\nabla\Phi(\0) = \kappa I_{d}$$ with $I_{d}$ the 
	$d \times d$ identity matrix.
	On the other hand
	\[
	\sup_{\ww\in \mathscr{B}} \|\nabla^2\Phi(\ww)\|_{p} = \sup_{\ww\in \mathscr{B}} \|\nabla^2 \phi(\ww)\|_{p} \leq \frac{\|\nabla\Phi(\0)\|_{p}}{C\alpha}= \frac{\kappa}{C\alpha}.
	\]
	Denote $B(\kappa I_{d},\kappa/(C\alpha))$ the $d \times d$ matrices with
	distance at most $\kappa/(C\alpha)$ from $\kappa I_{d}$, in terms of the maximum norm
	on $\mathbb{Q}_{p}^{d^{2}}$ (maximum norm $|.|_{p}$ of entries). 
	Since $\mathscr{B}$ has diameter $2\alpha$,
	it follows from the Mean Value Inequality (Theorem~\ref{pmvt}) applied to the gradient $\nabla\Phi$ that $\nabla\Phi(\ww) \in B(\kappa I_{d},2\alpha\cdot \kappa/(C\alpha))=B(\kappa I_{d},2\kappa/C)$ for all $\ww\in \mathscr{B}$. Thus by the Mean Value Inequality applied to $\Phi$, and the strong triangle inequality, for all $\yy,\ww\in \mathscr{B}$ 
	$$\|\Phi(\ww) - \Phi(\yy)\|_{p} \leq  \max\{ \|\kappa I_{d}\|_{p},2\kappa/C\}\|\ww-\yy\|_{p}=
	\max\{ \kappa,2\kappa/C\}\|\ww-\yy\|_{p}.$$ 
	
	The invertible $d\times d$ matrices form an open subset of $\mathbb{Q}_{p}^{d\times d}$ by Proposition~\ref{ppp1} and $\kappa>0$, so we see from the $p$-adic 
	Inverse Function Theorem~\ref{ift} that if $C>2$ is sufficiently large, then $B(\kappa I_{d},2\kappa/C)$ consists of
	invertible matrices. Hence $\Phi$ is bi-Lipschitz with a uniform bi-Lipschitz constant. Note
	\[
	S(\phi,\delta) = \Phi^{-1}\big(B_{d-1}(\0,\alpha)\times (-\delta\kappa,\delta\kappa)\big)
	\]
	and it is clear that the latter set $B_{d-1}(\0,\alpha)\times (-\delta\kappa,\delta\kappa)$ can be covered 
	by $$\asymp \kappa (\delta/\alpha)^{-(d-1)}=\kappa(\alpha/\delta)^{d-1}$$ balls of radius $\delta$.
	Since $\kappa$ is fixed the proof is finished
	via Proposition~\ref{ppp2}, when we allow the implied constant to  depend only on
	the bi-Lipschitz constant.
\end{proof}

We now use Lemma~\ref{ll} in each of the possible cases outlined by Lemma~\ref{h_lem} in order to finish the proof of Lemma~\ref{S(a)_haus_cont}.

\subsection*{\textbf{Case (a)}}
In this case observe that $S(\ba)$ is a $p$-adic $\Psi(\ba)$-thickened hyperplane (of $\Qp^{d}$) so can be covered by $\asymp \Psi(\ba)^{-(d-1)}$ balls of radius $\Psi(\ba)$. Hence
\begin{equation*}
\HH^{f}_{\infty}(S(\ba)) \ll \Psi(\ba)^{-(d-1)}f(\Psi(\ba)).
\end{equation*}

\subsection*{\textbf{Case (b)}}
By the assumption and conclusion of case (b) we have
\begin{align*}
\|\nabla h_{\ba}(\xx)\|_p &\gg \|\ba_{2}\|_{p}^{-1}\|\ba_{1}\|_{p} \\ &\geq \sup_{\ww \in K} \|\widetilde{a_{2}}\ba_{2}\cdot \nabla \textbf{g}(\ww)\|_{p}\\ &\gg_{\ba} 1 \\ &\gg \sup_{\ww \in K}\|\nabla^{2}h_{\ba}(\ww)\|_{p}.
\end{align*}
The final inequality is due to the observation from Lemma~\ref{h_lem} that $\|\nabla^{2}h_{\ba}(\xx)\|_{p}\ll 1$ for all $\xx \in K$, and the constant in the penultimate inequality is dependent on $$\widetilde{a_{2}}\ba_{2} \in \{\zz\in\Z^{n-d}: \|\zz\|_{p}=1\}$$
by \eqref{eq:Wnew}.
Since the supremum is taken over a compact set of $\widetilde{a_{2}}\ba_{2}$, we can choose a uniform lower bound. Note that this constant is strictly positive since $\widetilde{a_{2}}\ba_{2}\cdot \nabla \textbf{g}(\ww)=0$ for all $\ww\in K$ implies $\nabla^{2}h_{\ba}(\ww)=\0$, which we have removed. Thus, for suitably chosen $\epsilon>0$ dependent on the implied constants in the inequalities above and $C>0$ appearing in Lemma~\ref{ll}, let
\begin{equation*}
\alpha=\epsilon , \quad \delta=\frac{\Psi(\ba)}{\widetilde{a_{2}}\|\nabla h_{\ba}(\xx)\|_{p}}\asymp \frac{\Psi(\ba)}{\|\ba_{1}\|_{p}} , \quad \phi=h_{\ba}.
\end{equation*}
Hereby for the equivalence in $\delta$ we used that $\Vert\nabla h_{\ba}(\xx)\Vert_{p}\ll 1$ uniformly
as well since $\nabla h_{\ba}$ is continuous on 
the compact set $K$.
Then by Lemma~\ref{ll}
\begin{equation*}   
S(\ba) \cap B(\xx,\epsilon)
\end{equation*}
can be covered by 
\begin{equation*}
\asymp \epsilon^{(d-1)} \|\ba_{1}\|_{p}^{(d-1)} \Psi(\ba)^{-(d-1)}
\end{equation*}
balls of radius $\asymp \Psi(\ba) \|\ba_{1}\|_{p}^{-1}$. Thus
\begin{align*}
\HH^{f}_{\infty}(S(\ba)) & \ll \|\ba_{1}\|_{p}^{(d-1)} \Psi(\ba)^{-(d-1)} f\left( \frac{\Psi(\ba)}{\|\ba_{1}\|_{p}} \right) \\
& \hspace{-4ex} \overset{\text{condition (Ip)}}{\ll} \|\ba_{1}\|_{p}^{(d-1)-s} \Psi(\ba)^{-(d-1)} f\left( \Psi(\ba) \right) \quad \quad \quad (\text{ since } \|\ba_{1}\|_{p}^{-1}\geq1).
\end{align*}
Since $$\|\ba_{2}\|_{p}^{-1}\|\ba_{1}\|_{p}> \sup_{\ww\in K}\|\widetilde{a_{2}}\ba_{2}\cdot \nabla \textbf{g}(\ww)\|\gg 1$$ we have that $\|\ba_{1}\|_{p}\asymp \|(\ba_{1},\ba_{2})\|_{p}$, so
\begin{equation*}
\HH^{f}_{\infty}(S(\ba)) \ll  \|(\ba_{1},\ba_{2})\|_{p}^{(d-1)-s} \Psi(\ba)^{-(d-1)} f\left( \Psi(\ba) \right).
\end{equation*}
\subsection*{\textbf{Case (c)}}
Fix $k \in \Z$ and consider the $p$-adic annulus
\begin{equation*}
A_{k}=B_{d}(\vv, p^{-k}) \backslash B_{d}(\vv, p^{-(k+1)}).
\end{equation*}
Note that $\{A_{k}\}_{k \in \Z}$ partitions $\Qp^{d}\backslash \{\vv\}$ and since $K$ is bounded there exists some $k_{0} \in \Z$ such that
\begin{equation}\label{A_k_cover}
\bigcup_{k \geq k_{0}}A_{k} \supseteq K \backslash \{\vv\}.
\end{equation}
Observe that $\|\nabla h_{\ba}(\xx) \|_{p} \asymp \|\xx-\vv\|_{p} \asymp p^{-k}$ for all $\xx \in A_{k}$. Note that by Lemma~\ref{h_lem} 
\begin{equation*}
\|\nabla^{2}h_{\ba}(\xx)\|_{p} \ll 1.
\end{equation*}
Choose suitable $\epsilon>0$ such that
\begin{equation*}
\|\nabla h_{\ba}(\xx)\|_{p} \geq C(\epsilon p^{-k}) \sup_{\ww \in K} \|\nabla^2 h_{\ba}(\ww) \|_{p}
\end{equation*}
where $C>0$ comes from Lemma~\ref{ll}. 
 Letting 
\begin{equation*}
\alpha=\epsilon p^{-k}, \quad  \delta=\frac{\Psi(\ba)}{ \widetilde{a_{2}}\|\nabla h_{\ba}(\xx) \|_{p}}\asymp p^{k}\frac{\Psi(\ba)}{\widetilde{a_{2}}}, \quad \phi=h_{\ba}
\end{equation*}
then by Lemma~\ref{ll} we have that $$S(h,\delta)=S(\ba)\cap B_{d}(\xx,\epsilon p^{-k})$$ can be covered by
\begin{equation*}
\asymp \left(\frac{\alpha}{\delta}\right)^{d-1}\asymp \epsilon^{d-1} \widetilde{a_{2}}^{d-1}p^{-2k(d-1)} \Psi(\ba)^{-(d-1)}
\end{equation*}
balls of radius $\asymp p^{k}\frac{\Psi(\ba)}{\widetilde{a_{2}}}$. Observe that $B_{d}(\vv,p^{-k})\supseteq A_k$ can be covered by $p^{d}$ disjoint balls of radius $p^{-(k+1)}$. To see this, first take $\widetilde{\vv} \in \Q^{d}$ to be the rational
vector obtained from cutting off in any component the Hensel digits
from position $k+1$ onwards (which is
of the form $\widetilde{\vv}=N/p^{a}$ for some $N\in\Z^d, a\in \N_0$, with $p^a=\Vert \vv\Vert$, thus $a=0$ and $\widetilde{\vv}\in \Z^d$ if
$\vv\in \Z_p^d$). Then consider the $p^{d}$ balls
of radius $p^{-(k+1)}$ with centres $\widetilde{\vv}+tp^{k+1}$ for $t \in \{0, 1, \dots , p-1\}^{d}$. We remark that one of these balls is $B(\vv,p^{-(k+1)})$, as $\|\vv-\widetilde{\vv}+t\|\leq p^{-(k+1)}$ for some $t \in \{0, 1, \dots , p-1\}^{d}$ and in $p$-adic space if $x \in B(y,r)$ then $B(y,r)=B(x,r)$. So in fact we only require $p^{d}-1$ balls to cover $A_{k}$. Thus
\begin{equation*}
\HH^{f}_{\infty}(S(\ba)\cap A_{k})   \ll (p^{d}-1)p^{- 2k(d-1)} \widetilde{a_{2}}^{(d-1)} \Psi(\ba)^{-(d-1)} f\left( p^{k}\frac{\Psi(\ba)}{\widetilde{a_{2}}}\right).
\end{equation*}
So, by \eqref{A_k_cover},
\begin{align*}
\HH^{f}_{\infty}(S(\ba))  & \ll  \widetilde{a_{2}}^{(d-1)} \Psi(\ba)^{-(d-1)} \sum_{k \geq k_{0}} p^{- 2k(d-1)} f\left( p^{k}\frac{\Psi(\ba)}{\widetilde{a_{2}}}\right) \\
& \overset{\text{ condition (Ip) }}{\ll}  \widetilde{a_{2}}^{(d-1)-s} \Psi(\ba)^{-(d-1)} f\left(\Psi(\ba)\right) \sum_{k \geq k_{0}} p^{-2k(d-1)+ks} \\
& \overset{s<2(d-1)}{\ll}\widetilde{a_{2}}^{(d-1)-s}\Psi(\ba)^{-(d-1)} f\left( \Psi(\ba)\right).
\end{align*}
Note that in order for such $\vv\in K$ appearing in case (c) to exist we have that
\begin{equation*}
\nabla h_{\ba}(\vv)=0 \quad \implies \quad \widetilde{a_{2}}\ba_{1}=\widetilde{a_{2}}\ba_{2}\cdot \nabla \textbf{g}(\vv).
\end{equation*}
By the strong triangle inequality, we can deduce that
\begin{equation*}
\|\ba_{1}\|_{p}=\|\ba_{2}\cdot \nabla \textbf{g}(\vv)\|_{p}\le \|\ba_{2}\|_{p}\cdot \|\nabla \textbf{g}(\vv)\|_{p} \ll \|\ba_{2}\|_{p},
\end{equation*}
and so $\widetilde{a_{2}}\overset{\text{def}}{=}\|\ba_{2}\|_{p}\asymp \|(\ba_{1},\ba_{2})\|_{p}$. Hence
\begin{equation*}
\HH^{f}_{\infty}(S(\ba)) \ll \|(\ba_{1},\ba_{2})\|_{p}^{(d-1)-s}\Psi(\ba)^{-(d-1)} f\left( \Psi(\ba)\right).
\end{equation*}

\subsection*{\textbf{Case (d)}}
Since $\|\nabla h_{\ba}(\xx)\|_{p}\asymp 1$ we have that, for suitably chosen $C, \epsilon>0$,
\begin{equation*}
\|\nabla h_{\ba}(\xx)\|_{p}\geq C \epsilon \sup_{\ww \in K} \|\nabla^{2} h_{\ba}(\ww)\|_{p}.
\end{equation*}
Applying Lemma~\ref{ll} with
\begin{equation*}
\alpha=\epsilon, \quad \delta=\frac{\Psi(\ba)}{\widetilde{a_{2}}\|\nabla h_{\ba}(\xx)\|_{p}}\asymp\frac{\Psi(\ba)}{\widetilde{a_{2}}}, \quad \phi=h_{\ba}
\end{equation*}
we have that $S(\ba)$ can be covered by 
\begin{equation*}
\asymp \epsilon^{(d-1)} \widetilde{a_{2}}^{(d-1)} \Psi(\ba)^{-(d-1)}
\end{equation*}
balls of radius $\asymp \frac{\Psi(\ba)}{\widetilde{\ba_{2}}}$. Thus
\begin{align*}
\HH^{f}_{\infty}(S(\ba)) & \ll \widetilde{a_{2}}^{(d-1)} \Psi(\ba)^{-(d-1)} f\left( \frac{\Psi(\ba)}{\widetilde{a_{2}}}\right) \\
&\overset{\text{condition (Ip)}}{\ll} \widetilde{a_{2}}^{(d-1)-s}\Psi(\ba)^{-(d-1)} f(\Psi(\ba)).
\end{align*}
Note that in this case we have that $$\|\ba_{1}\|_{p} \leq \sup_{\ww \in K}\|\ba_{2}\cdot \nabla g(\ww) \|_{p} \ll \|\ba_{2}\|_p\overset{\text{def}}{=}\widetilde{a_{2}}$$ and so $\widetilde{a_{2}}\asymp \|(\ba_{1}, \ba_{2})\|_{p}$. Thus
\begin{equation*}
\HH^{f}_{\infty}(S(\ba))\ll \|(\ba_{1},\ba_{2})\|_{p}^{(d-1)-s}\Psi(\ba)^{-(d-1)} f(\Psi(\ba)).
\end{equation*}

Combining the outcomes of the above four cases,  we have that for any $\ba=(a_{0},a_{1}, \dots , a_{n})=(a_{0}, \ba_{1}, \ba_{2}) \in \Z\times(\Z^{n}\backslash\{\0\})$, \begin{equation*}
\HH^{f}_{\infty}(S(\ba)) \ll \|(\ba_{1},\ba_{2})\|_{p}^{(d-1)-s} \Psi(\ba)^{-(d-1)}f(\Psi(\ba)).
\end{equation*}
The proof of Lemma~\ref{S(a)_haus_cont} is complete.
\qed

\section{Completing the proof of Theorem~\ref{dual_hypersurface_inhomogeneous}} \label{complete}
Note that for any subset of integers $Z \subseteq \Z^{n+1}\backslash \{\0\}$ we have
\begin{equation*}
E=E(Z)=\{\xx \in K : (\xx,g(\xx)) \in D_{n,p}^{\theta}(\Psi, Z)\}=\limsup_{\ba \in Z}S(\ba),
\end{equation*}
The map $G: \xx \mapsto (\xx,g(\xx))$ is bi-Lipschitz since $K$ is bounded, and so
\begin{equation*}
\HH^{f}(G(E)) \asymp \HH^{f}(D_{n,p}^{\theta}(\Psi,Z)\cap \{(\xx,\textbf{g}(\xx)): \xx \in K\}) \asymp  \HH^{f}(E).
\end{equation*}

Now, by the Hausdorff-Cantelli Lemma, we have that
\begin{equation*}
\HH^{f}(E)=0 \quad \text{ if } \quad \sum_{\ba \in Z} \HH^{f}_{\infty}(S(\ba))< \infty,
\end{equation*}
and so 
\begin{equation*}
\HH^{f}(D_{n,p}^{\theta}(\Psi,Z)\cap \MM) =0 \quad \text{ if } \quad  \sum_{\ba \in Z}  \HH^{f}_{\infty}(S(\ba))< \infty.
\end{equation*}
By Lemma~\ref{S(a)_haus_cont} the above summation can be written as
\begin{equation} \label{sum}
\sum_{\ba \in \{Z: (\ba_{1},\ba_{2}) \neq \0\}} \|(\ba_{1},\ba_{2})\|_{p}^{(d-1)-s} \Psi(\ba)^{-(d-1)}f(\Psi(\ba)) \, \, + \, \, \sum_{\ba \in \{Z: (\ba_{1},\ba_{2}) = \0\}} \Psi(\ba)^{-(d-1)}f(\Psi(\ba)).
\end{equation}
Since
\begin{equation*}
\sum_{\ba \in Z} \Psi(\ba)^{-(d-1)}f(\Psi(\ba))<\infty
\end{equation*}
by assumption (in whichever case of $Z$ is chosen) the second summation is convergent, so it remains to prove convergence of the first sum. This summation can be rewritten as
\begin{equation*}
\sum_{k \in \N_{0}} \sum_{\ba \in Z(p,k) } p^{k(s-(d-1))}\Psi(\ba)^{-(d-1)}f(\Psi(\ba)),
\end{equation*}
where 
\begin{equation*}
Z(p,k)= \{(a_{0}, \dots , a_{n}) \in Z : p^{k}| a_{i}  \, \forall \, 1\leq i \leq n \, \, \& \, \,  \exists \, 1\leq j \leq n \, \, p^{k+1} \not| a_{j}\}.
\end{equation*}
Note that the index $0$ is excluded from the divisibility conditions, so in other words
$$Z(p,k)=\{(a_{0},\dots, a_{n})\in Z: \|(\ba_{1},\ba_{2})\|_{p}=p^{-k}\}.$$ We will see that we can
 restrict ourselves to finite partial sums in all three cases.
In the following arguments we will frequently implicitly use the well-known fact that
\[
|a+b|_p = \max\{ |a|_p , |b|_p \}, \qquad \text{if}\; |a|_p\ne |b|_p.
\]

 Now let us consider the cases in our distinction one by one:
\begin{itemize}

\item[(i)] $Z=Z(1)$ (pairwise coprime) and inhomogeneous: Then $\|(\ba_{1},\ba_{2})\|_{p}=1$, and so \eqref{sum} becomes the usual summation
as in (i), which is convergent by assumption. \\ 

\item[(ii)] $Z=Z(2)$ (coprime): If $p|a_{0}$ then there exists some $1\leq i \leq n$ such that $p \not| a_{i}$, and so  $\|(\ba_{1},\ba_{2})\|_{p}=1$. Thus convergence of the summation is immediate by assumption. Henceforth assume $|a_{0}|_{p}=1$. 
We may assume there exists $\ell \in \N_{0}$ such that 
\begin{equation}  \label{eq:tor}
\|(\xx,\textbf{g}(\xx))\|_{p}\leq p^{\ell} \quad \forall \,\, \xx \in K,
\end{equation}
since this is true for any compact subset of $K$ and by sigma-additivity of measures.
First assume $\theta=0$.
The strong triangle inequality, and the fact that $\Psi(\ba)<1$ for sufficiently large $\|\ba\|$, implies that for $S(\ba)$ to be non-empty we at least require
\begin{equation*}
|\ba_{1}\cdot \xx + \ba_{2} \cdot \textbf{g}(\xx)|_{p}=|a_{0}|_{p}=1.
\end{equation*}
If $\|(\ba_{1},\ba_{2})\|_{p}\leq p^{-\ell-1}$ then this cannot be true for any $\xx\in K$ since
\[
|\ba_{1}\cdot \xx + \ba_{2} \cdot \textbf{g}(\xx)|_{p}\le
 \Vert (\ba_1, \ba_2)\Vert_p \cdot \Vert (\xx,\textbf{g}(\xx))\Vert_p \le p^{-1}<1
\]
so we must have that $\|(\ba_{1},\ba_{2})\|_{p}> p^{-\ell-1}$,
hence $\|(\ba_{1},\ba_{2})\|_{p}\ge p^{-\ell}$.
Thus, in the cover of $E$ we only need to consider $S(\ba)$ with $\|(\ba_{1},\ba_{2})\|_{p}=p^{-k}$ for $k=0,\dots,\ell$. That is, we only need to show convergence of the summation
\begin{equation} \label{finite_sums}
\sum_{k=0}^{\ell} \sum_{\ba \in Z(p,k) } p^{k(s-(d-1))}\Psi(\ba)^{-(d-1)}f(\Psi(\ba)).
\end{equation}
Note that for each $k=0, \dots , \ell$ the inner sum is convergent (by assumption), so since the outer sum is finite we indeed have convergence.

Now assume $\theta\ne 0$ and $|\theta|_p\neq 1$. Let $p^{y}:= \vert \theta|_p$ for an integer $y<0$ (note we assume $\theta\in \Q_p$. We can assume \eqref{eq:tor}
so that if $\Vert (\ba_1,\ba_2)\Vert_p\le p^{-\ell-1}$ we have
 \begin{equation} \label{letztes}
 \Vert (\ba_1, \ba_2) \cdot (\xx,\textbf{g}(\xx)) + a_0 + \theta\Vert_p  \le 
 \max\{ \Vert (\ba_1,\ba_2)\Vert_p \cdot 
 \Vert (\xx,\textbf{g}(\xx))\Vert_p, |a_{0}|_{p}, |\theta|_p \}
 \end{equation}
 and there is in fact equality. To see this observe that 
 \[
 \Vert (\ba_1,\ba_2)\Vert_p \cdot 
 \Vert (\xx,\textbf{g}(\xx))\Vert_p\le p^{-\ell-1}p^{\ell}= p^{-1}<1,
 \] 
moreover $p^{y}<1$ as we noticed $y<0$, and $|a_{0}|_{p}=1$.
Hence
 \[
 \Vert (\ba_1, \ba_2) \cdot (\xx,\textbf{g}(\xx))+ a_0+ \theta \Vert_p=\max\{1,p^{y}\}=1
 \]
 is absolutely bounded from below uniformly on $K$. Since $\Psi\to 0$
 there are only finitely many solutions (if any) $\ba$ to $\Vert (\ba_1, \ba_2) \cdot (\xx,\textbf{g}(\xx))+ a_0+ \theta \Vert_p < \Psi(\ba)$. This argument
 again shows we may restrict to  $\Vert (\ba_1, \ba_2)\Vert_p\ge p^{-\ell}$ leading to a finite sum
 \[
 \sum_{k=0}^{  \ell } \sum_{\ba \in Z(p,k) } p^{k(s-(d-1))}\Psi(\ba)^{-(d-1)}f(\Psi(\ba)),
 \]
which converges for similar reasons as the sum for $\theta=0$. \\

\item[(iii)] $Z=\Z^{n+1}\backslash \{\0\}$, $\theta=0$, and $\Psi$ satisfies $p\Psi(\ba)\le \Psi(p^{-1}\ba)$: Again, we may assume that there exists $\ell \in \N$ such that \eqref{eq:tor} holds.
 Suppose that $0\neq \|(\ba_{1},\ba_{2})\|_{p}\leq p^{-\ell-1}$. We claim that in order for $S(\ba)$ to be non-empty we need $p|a_{0}$. From this we will deduce that $S(p^{-1}\ba)\supseteq S(\ba)$. \par 
By \eqref{eq:tor} we have
\begin{align*}
| \ba_{1}\cdot \xx+\ba_{2}\cdot\textbf{g}(\xx) + a_{0}|_{p} &\leq \max \left\{| \ba_{1}\cdot \xx+\ba_{2}\cdot\textbf{g}(\xx)|_{p}, |a_{0}|_{p}\right\}\\ &\le
\max\{ \|(\ba_{1},\ba_{2})\|_{p} \cdot \|(\xx,\boldsymbol{g}(\xx))\|_{p}, |a_0|_p\}\\ &\le \max\{ p^{-\ell-1}p^{\ell} , |a_0|_p\}\\ &\le \max\{ p^{-1},|a_0|_p \}.
\end{align*}
There is equality in the above inequalities (apart from possibly the third) if $|a_0|_p=1>p^{-1}$, leading us in any case to conclude that $| \ba_{1}\cdot \xx+\ba_{2}\cdot\textbf{g}(\xx) + a_{0}|_{p}=1<\Psi(\ba)$ which cannot be true for sufficiently large $\|\ba\|$. Thus $|a_{0}|_{p}<1$ and so $p|a_{0}$. Hence $p^{-1}\ba=(\ba_{1}',\ba_{2}', a_{0}') \in \Z^{n+1}\backslash \{\0\}$ and if $\xx \in S(\ba)$ then
\begin{equation*}
|\ba_{1}'\cdot \xx+\ba_{2}'\cdot\textbf{g}(\xx)+ a_{0}'|_{p}<p \Psi(\ba)\le \Psi(p^{-1}\ba)
\end{equation*}
thus $\xx \in S(p^{-1}\ba)$, and so $S(\ba)\subseteq S(p^{-1}\ba)$. 
Since $D_{n,p}(\Psi)$ is defined as the limsup set of $S(\ba)$ over integer
vectors $\ba$,
this argument implies that we only need to consider integer points $\ba$ with $\|(\ba_{1},\ba_{2})\|_{p}\ge p^{-\ell}$. 
This again leads us to the convergent sum \eqref{finite_sums}. \par
Lastly, if $(\ba_{1},\ba_{2})=\0$ then for $S(\ba)$ to be non-empty we have that
\begin{equation*}
    |a_{0}|_{p}<\Psi(\ba)<|a_{0}|^{-1}\, ,
\end{equation*}
which cannot be true, so in this case $S(\ba)=\emptyset$.


\end{itemize}

\providecommand{\bysame}{\leavevmode\hbox to3em{\hrulefill}\thinspace}
\providecommand{\MR}{\relax\ifhmode\unskip\space\fi MR }
\providecommand{\MRhref}[2]{%
  \href{http://www.ams.org/mathscinet-getitem?mr=#1}{#2}
}
\providecommand{\href}[2]{#2}

\end{document}